\documentclass[10pt]{article}

\usepackage{amsfonts}
\usepackage{amsmath}
\usepackage{amsthm}
\newtheorem{theorem}{Theorem}
\newtheorem{lemma}[theorem]{Lemma}
\newtheorem{q}[theorem]{Question}
\newtheorem{conj}[theorem]{Conjecture}

\def\A{\mathcal{A}}
\def\B{\mathcal{B}}
\def\U{\mathcal{U}}
\def\C{\mathcal{C}}
\def\T{\mathcal{T}}

\def\eps{\varepsilon}

\def\P{\mathbb{P}}

\def\ps{\mathcal{P}(X)}

\begin{document}

\title{Set Systems Containing Many Maximal Chains} 
\author{J.~Robert Johnson\thanks{School of Mathematical Sciences, Queen Mary University of London, London E1 4NS}, Imre Leader\thanks{Department of Pure Mathematics and Mathematical Statistics, Centre for Mathematical Sciences, Wilberforce Road, Cambridge CB3 0WB} and Paul Russell$^\dagger$} 

\maketitle

The purpose of this note is to raise an extremal question on set systems (that is, subsets of the power set $\ps$, where $X=\{1,2,\dots,n\}$) which seems to be natural and appealing. Our question is: which set systems of a given size maximise the number of $(n+1)$-element chains in $\ps$? We refer to these chains as \emph{maximal chains} -- we emphasise that they are maximal in $\ps$ rather than just maximal within the set system.

\begin{q}\label{chain:q}
For given $|\A|$, which $\A\subseteq\ps$ contains the most maximal chains?
\end{q}

For $\A\subseteq\ps$ we denote the number of maximal chains in $\A$ by $c(\A)$. We will show that for each fixed $\alpha>0$ we have
\[
\max\{c(\A): |A|=\alpha 2^n\}=(\alpha+o(1))\;n!
\]
For smaller set systems we are unable to answer the question. We conjecture that a `tower of cubes' construction (defined later) is extremal. We finish by mentioning briefly a connection to an extremal problem on posets and a variant of our question for the grid graph.

This question does not appear to have been asked before, although there has been work on the problem of maximising the number of $k$-element chains for smaller $k$. For $k=2$ (that is, maximising the number of comparable pairs) Alon and Frankl \cite{AF} proved that a tower of cubes is approximately extremal. They also proved a similar but less exact result for arbitrary $k$ (with $k$ fixed and $n$ large). By contrast, in this note we are interested in the case when $k$ is as large as it can be.  

We note briefly that the question of \emph{minimising} the number of $k$-element chains has also been studied. In this direction Kleitman \cite{K} proved the following quantitative extension of Sperner's theorem: the minimum number of comparable pairs is achieved by a union of middle layers of $\ps$. He also conjectured that the same construction minimises the number of $k$ element chains for any $k$ -- this problem remains open. See also the recent work of Dove, Griggs, Kang and Sereni \cite{DGJS}.

Our question also has some similarity with a problem on `most probably intersecting systems' posed by Katona, Katona and Katona \cite{KKK}.
Answering their question, Russell \cite{R} determined (approximately) which set systems of a given size contain the maximum number of intersecting families of size $k$.

It is easy to show that a family containing a fixed proportion of all sets can contain no more than the same fixed proportion of all maximal chains (this is the content of the next easy lemma). It turns out that this is asymptotically best possible, although the construction proving this is perhaps unexpected. 

\begin{lemma}
If $\A\subseteq\ps$ with $|\A|=\alpha 2^n$ then $c(\A)\leq\alpha \; n!$\;.
\end{lemma}

\begin{proof}
If $|\A|=\alpha 2^n$ then by averaging there is some $r$ with $|\A\cap X^{(r)}|\leq \alpha\binom{n}{r}$. The number of maximal chains in $\ps$ that contain a point from $\A\cap X^{(r)}$ is therefore at most $\alpha \; n!$\;.
\end{proof}

To show that this simple bound is asymptotically tight it will be convenient to formulate the question slightly differently. Maximal chains in $\ps$ are in one-to-one correspondence with permutations of $X$, with the permutation $a_1a_2\dots a_n$ corresponding to the chain
\[
\emptyset,\{a_1\},\{a_1,a_2\},\dots,X.
\]
For a permutation $\sigma=a_1a_2\dots a_n\in S_n$ we denote by $I(\sigma)$ the set of subsets of $X$ that lie in the maximal chain in $\ps$ corresponding to the permutation $\sigma$ (that is, those of the form $\{a_1,a_2,\dots,a_r\}$). In
this language, Question \ref{chain:q} is then equivalent to the following
question. 

\begin{q}\label{perm:q}
For given $|\B|$, which $\B\subseteq S_n$ minimises $\left|\bigcup_{\sigma\in\B} I(\sigma)\right|$?
\end{q}

We now show that the bound in Lemma 2 is asymptotically correct.

\begin{theorem}\label{main}
For any $0<\alpha<1$ and $\eps>0$ there exists $n_0$ such that for $n>n_0$ there is a family $\A\subseteq\ps$ with $|\A|\leq(\alpha+\eps)2^n$ and $c(\A)\geq\alpha \; n!$\;.
\end{theorem}

\begin{proof}
It will be tidier to prove this by constructing, for any $\eps>0$, a family $\B$ of at least $(\alpha-\eps) \; n!$ permutations such that
\[
\Big|\bigcup_{\sigma\in\B} I(\sigma)\Big|\leq(\alpha+\eps)2^n.
\]
This clearly implies the result.

Given $\eps>0$, choose $k$ so that $\frac{1}{2^k}<\frac{\eps}{2}$. Let $\U$ be an up-set in $\mathcal{P}(\{1,\dots,k\})$ with $|\U|=\beta 2^k$ where $\alpha<\beta \leq\alpha+\frac{\eps}{2}$. Let $\B_p$ be the set of all permutations of $X$ for which the set of elements of $[k]$ appearing in the first $pn$ positions of the permutation is an element of $\U$. We will show that there is a $p$ satisfying $p\leq\frac{1}{2}-c_k$, where $c_k$ is a positive constant depending on $k$ but not on $n$, with $|\B_p|\geq(\alpha-\eps)\;n!$\;. Since any set $S$ of size greater than $pn$ is in $\bigcup_{\sigma\in\B_p} I(\sigma)$ if and only if $S\cap[k]\in\U$, it will follow that 
\[
\Big|\bigcup_{\sigma\in\B_p} I(\sigma)\Big|\leq|\{S: S\cap[k]\in\U\}|+|\{S: |S|\leq pn\}|=\beta 2^n+\sum_{i=0}^{pn}\binom{n}{i}<(\alpha+\eps)2^n.
\]

It remains to prove that such a $p$ exists. The probability that a randomly chosen permutation $\sigma$ is in $\B_p$ is
\[
\P(\sigma\in\B_p)=(1-O(1/n))f(p)
\]
where $f(p)=\sum_{X\in\U}p^{|X|}(1-p)^{k-|X|}$. The polynomial $f(p)$ does not involve $n$, is increasing in $p$ (since $\U$ is an up-set), and satisfies $f(\frac{1}{2})=\beta$. It follows that there is some positive constant $c_k$ with $f(\frac{1}{2}-c_k)=\alpha$ and that if we set $p=\frac{1}{2}-c_k$ then 
\[
\P(\sigma\in\B_p)\rightarrow\alpha.
\]
So for $n$ sufficiently large we have that $|\B_p|\geq(\alpha-\eps)\;n!$\;, as required.

\end{proof}

When $|\A|=o(2^n)$ the situation seems to be more complicated. We now describe a construction of a set system that we believe is a plausible candidate for being extremal.

Suppose that $n=tk$ where $t,k\in\mathbb{N}$. Let $X_i=\{(i-1)t+1,(i-1)t+2,\dots,it\}$ for $1\leq i\leq k$. and define the \emph{tower of $t$-cubes} set system by
\[
\T_t=\{A:(X_1\cup\dots\cup X_s)\subseteq A\subseteq(X_1\cup\dots\cup X_{s+1}) \text { for some } 0\leq s\leq k-1\}.
\]
Thus $|\T_t|=\frac{n}{t}2^{t}-\frac{n}{t}+1$ and $c(\T_k)=(t!)^{n/k}$.

This construction is easily seen to be extremal when $t=2$. Indeed, if we take a fixed maximal chain $\C$ in $\A$ and $X\subseteq(\A\setminus\C)$ then there is at most one maximal chain $\C'\subseteq\A$ with $\C'\setminus\C=X$. It follows that $c(\A)\leq 2^{|\A|-(n+1)}$. 

We conjecture that this construction is extremal for all $t$. 

\begin{conj}\label{tower:conj}
If $|\A|=|\T_t|$ then $c(\A)\leq c(\T_t)$.
\end{conj}

The extreme cases $t=3$ and $t=n/2$ are both particularly appealing.

Of course the `tower of cubes' construction only gives set systems of particular sizes. We can generalise the construction to include a wider range of possible sizes by allowing cubes of dimension $t$ and $t+1$. We say that a family $\T$ is a \emph{generalised tower of cubes} if it is of the form:
\[
\T=\{A:(X_1\cup\dots\cup X_s)\subseteq A\subseteq(X_1\cup\dots\cup X_{s+1}) \text { for some } 0\leq s\leq k-1\}
\]
where $X_1,\dots,X_k$ are pairwise disjoint subsets of $X$ whose union is $X$ and with $|X_i|=t$ or $t+1$ for all $i$. 

If $n=a(t-1)+bt$ then there is a generalised tower of cubes with $|\T|=a(2^t-1)+b(2^{t+1}-1)+1$ and $c(\T)=t!^a(t+1)!^b$. We conjecture that, for values of $|\A|$ which allow it, a construction of this form is extremal. 

\begin{conj}\label{gen-tower:conj}
If $|\A|=|\T|$, where $\T$ is a generalised tower of cubes, then $c(\A)\leq c(\T)$.
\end{conj}

We do not even have a good conjecture for the extremal set systems when $|\A|$ does not allow a `generalised tower of cubes' construction.
In particular, we have no idea what happens when $|\A|$ is between $2^{n/2}$ and $\alpha2^n$. In general, the situation may be quite complicated: for example, when $|\A|=2\alpha2^{n/2}$ we could take a tower of cubes comprising two cubes each of dimension $n/2$, and replace each of these by an $\alpha2^{n/2}$-size family as constructed in Theorem \ref{main}.

We briefly mention that a first step in proving extremal results is often to apply some kind of compression operation. See \cite{BB} for numerous examples of this and background material on compressions. Standard $ij$-compressions can be used to simplify our set system without decreasing the number of maximal chains, as we now describe. 

Let $C_{ij}$ denote the $ij$-compression operation. That is for a set $A\in\ps$,
\[
C_{ij}(A)=\left\{\begin{array}{ll}
A\setminus\{j\}\cup\{i\} & \text{ if } i\not\in A, j\in A,\\
A & \text{ otherwise}.\end{array}
\right.
\]
and for a set system $\A\subseteq\ps$,
\[
C_{ij}(\A)=\{C_{ij}(A):A\in\A\}\cup\{A:A\in\A,C_{ij}(A)\in\A\}.
\]

\begin{lemma}
If $\A\subseteq\ps$ then the number of maximal chains in $C_{ij}(\A)$ is at least the number of maximal chains in $\A$.
\end{lemma}

\begin{proof}
Let $\A'=C_{ij}(\A)$. Suppose that the chain corresponding to the permutation $\sigma=a_1\dots,a_n$ is in $\A$. If $i$ comes before $j$ in $\sigma$ then this chain is present in $\A'$. If $j$ comes before $i$ in $\sigma$ then either $\sigma'$, the chain corresponding to $\sigma$ with $i$ and $j$ exchanged, is in $\A'$ but not $\A$ or both $\sigma$ and $\sigma'$ are in $\A$ and $\A'$.  
\end{proof}

Applying $ij$-compressions repeatedly with $i<j$ allows us to assume that our set system is left-compressed: we have $C_{ij}(\A)=\A$ for all $i<j$. Unfortunately, it seems hard to make good use of this assumption. One simple consequence is that we can write down the best set system under the extra condition that each layer (except the top and bottom) contains exactly two sets (that is, $|\A\cap X^{(r)}|=2$ for all $1\leq r\leq n-1$). An easy calculation shows that the tower of 2 and 3 dimensional cubes of the same size contains more maximal chains. This is perhaps weak evidence that tower of cubes constructions are plausible candidates for extremal systems.

We turn now to a connection with posets. Given a poset $P$ on $X$, there is a natural way of identifying a linear extension of $P$ with a maximal chain in $\ps$. The collection of all down-sets in $P$ gives a family in $\ps$ which contains exactly those maximal chains that correspond to linear extensions of $P$. Since down-sets and antichains in a poset are in one-to-one correspondence this leads to the following question.

\begin{q}\label{poset}
What is the maximum number of linear extensions of a poset on $X$ that contains at most $m$ antichains?
\end{q} 

Clearly not all subsets of $\ps$ arise from posets in this way, but the towers of cubes and their generalisations described earlier do. It follows that a proof of Conjectures \ref{tower:conj} and \ref{gen-tower:conj} (that the `tower of cubes' construction is extremal for maximal chains in the cube) would answer Question \ref{poset} for appropriate values of $m$.

Finally, we mention a generalisation of Question \ref{chain:q} to grids. Since $\ps$ can be regarded as $\{0,1\}^n$, it is natural to ask analogous questions for more general products. A \emph{maximal chain} in $[k]^n=\{1,\dots,k\}^n$ consists of a sequence of $n(k-1)+1$ points of $[k]^n$ such that each point is obtained from its predecessor by adding 1 to one coordinate. Note that these maximal chains can also be regarded as shortest paths betwen $(1,\dots,1)$ and $(k,\dots,k)$ in the grid graph, just as maximal chains in $\ps$ correspond to shortest paths between $(0,\dots,0)$ and $(1,\dots,1)$ in the hypercube graph. Contrasting with the $\ps$ case where $k$ is fixed and $n$ is large, it is natural to consider the case of fixed dimension and large $k$. This seems to be an interesting question even for $n=2$.

\begin{q}
For given $|\A|$, which $\A\subseteq[k]^2$ contains the most maximal chains?
\end{q}

\end{document}